\documentclass[12pt,leqno]{article}
\usepackage{amsfonts}
\usepackage{amsmath, amsthm, amsfonts, amssymb, color}
\usepackage{mathrsfs}
\newtheorem{thm}{Theorem}[section]

\newtheorem{lem}[thm]{Lemma}

\theoremstyle{definition}

\newcommand{\scr}[1]{\mathscr #1}
\definecolor{wco}{rgb}{0.5,0.2,0.3}

\numberwithin{equation}{section} \theoremstyle{remark}

\newcommand{\ua}{\uparrow}

\title{{\bf Transportation Cost Inequalities for Neutral Functional Stochastic Equations}\footnote{Supported in
 part by  Lab. Math. Com. Sys., NNSFC(11131003), SRFDP and the Fundamental Research Funds for the Central Universities.}
}
\author{
{\bf  Jianhai Bao$^{b)}$,  Feng-Yu Wang$^{a),b)}$,  Chenggui Yuan$^{b)}$}\\
\footnotesize{$^{a)}$School of Mathematical Sciences, Beijing Normal
University, Beijing 100875, China}\\
 \footnotesize{$^{b)}$Department of Mathematics,
Swansea University, Singleton Park, SA2 8PP, UK}\\
\footnotesize{wangfy@bnu.edu.cn, F.-Y.Wang@swansea.ac.uk,
C.Yuan@swansea.ac.uk}}
\begin{document}
\def\R{\mathbb R}  \def\ff{\frac} \def\ss{\sqrt} \def\B{\mathbf
B}
\def\N{\mathbb N} \def\kk{\kappa} \def\m{{\bf m}}
\def\dd{\delta} \def\DD{\Dd} \def\vv{\varepsilon} \def\rr{\rho}
\def\<{\langle} \def\>{\rangle} \def\GG{\Gamma} \def\gg{\gamma}
  \def\nn{\nabla} \def\pp{\partial} \def\EE{\scr E}
\def\d{\text{\rm{d}}} \def\bb{\beta} \def\aa{\alpha} \def\D{\scr D}
  \def\si{\sigma} \def\ess{\text{\rm{ess}}}
\def\beg{\begin} \def\beq{\begin{equation}}  \def\F{\scr F}
\def\Ric{\text{\rm{Ric}}} \def\Hess{\text{\rm{Hess}}}
\def\e{\text{\rm{e}}} \def\ua{\underline a} \def\OO{\Omega}  \def\oo{\omega}
 \def\tt{\tilde} \def\Ric{\text{\rm{Ric}}}
\def\cut{\text{\rm{cut}}} \def\P{\mathbb P} \def\ifn{I_n(f^{\bigotimes n})}
\def\C{\scr C}      \def\aaa{\mathbf{r}}     \def\r{r}
\def\gap{\text{\rm{gap}}} \def\prr{\pi_{{\bf m},\varrho}}  \def\r{\mathbf r}
\def\Z{\mathbb Z} \def\vrr{\varrho} \def\l{\lambda}
\def\L{\scr L}\def\Tt{\tt} \def\TT{\tt}\def\II{\mathbb I}
\def\i{{\rm in}}\def\Sect{{\rm Sect}}\def\E{\mathbb E} \def\H{\mathbb H}
\def\M{\scr M}\def\Q{\mathbb Q} \def\texto{\text{o}} \def\LL{\Lambda}
\def\Rank{{\rm Rank}} \def\B{\scr B} \def\i{{\rm i}} \def\HR{\hat{\R}^d}
\def\to{\rightarrow}\def\l{\ell}\def\ll{\lambda}
\def\8{\infty}\def\ee{\epsilon}

\maketitle

\begin{abstract}

By  using   Girsanov transformation and martingale representation,
  Talagrand-type transportation cost inequalities, with respect to both the uniform and the $L^2$ distances  on the global free path space, are established for
the segment process associated to  a class of   neutral functional stochastic differential
equations. Neutral  functional stochastic partial differential
equations are also investigated.
\\

 \end{abstract}

\section{Introduction}
Let $(E,\B(E))$ be a  measurable space with $\rr$ a symmetric
non-negative measurable function on $E\times E$.   For any $p\ge1$
and  probability measures $\mu$ and $\nu$ on $(E,\mathscr{B}(E))$,
the $L^p$-transportation cost (or, the $L^p$-Wasserstein distance
if $\rr$ is a distance) induced by $\rr$   between these two
measures   is defined by
\begin{equation*}
W_{p,\rr}(\mu,\nu)=\inf_{\pi\in\C(\mu,\nu)}\Big\{\int_{E\times
E}\rr^p(x,y)\pi(\d x,\d y)\Big\}^{1/p},
\end{equation*}
where $\mathscr{C}(\mu,\nu)$ denotes the space of all couplings of
$\mu$ and $\nu$. In many practical situations, one wants to find
reasonable and simple  upper bounds for
 $W_{p,\rr}(\mu,\nu)$, where a fully satisfactory one
is given by the transportation cost inequality first found by  Talagrand  \cite{T96} for the standard Gaussian measure   $\mu$ on $\R^d$:
\begin{equation*}
W_{2,\rr}(\mu,f\mu)^2\le2 \mu(f\log f),\ \ f\ge 0, \mu(f)=1
\end{equation*}
with $\rr(x,y):=|x-y|.$  Since then, this type transportation
cost inequality has been intensively investigated and applied for
various different distributions. The importance of the study lies on
intrinsic links of the transportation cost inequality to several
crucial subjects, such as functional inequalities, concentration
phenomena, optimal transport problem, and large deviations, see e.g.
\cite{BG99, BGL, GRS11, FS07, L, OV, U, V03, W04b} and references
within.

In the past decade, a plenty of results have been published
concerning Talagrand-type transportation cost inequalities on the
path spaces of stochastic processes, see e.g. \cite{DGW04, WZ04,
WZ06} for diffusion processes on $\R^d$,  \cite{P11} for
multidimensional semi-martingales, \cite{U10}  for   diffusion
processes with history-dependent drift, \cite{W04, W04b} for
diffusion processes on Riemannian manifolds,  \cite{W10}   for SDEs
driven by
  pure jump processes, and   \cite{M10}
for SDEs driven by both Gaussian and jump noises. Recently,
transportation cost inequalities for the reflecting diffusion
processes on manifolds with boundary have been used in \cite{W11} to
characterize the curvature of the generator and the convexity of the
boundary.

Moreover, many different arguments have been developed to establish
the transportation cost inequality.  Among others,  the Girsanov
transformation  argument introduced  in \cite{DGW04} has been
efficiently applied, see  e.g. \cite{WZ06} for infinite-dimensional
dynamical systems, \cite{P11} for time-inhomogeneous diffusions,
\cite{U10} for
  multi-valued SDEs and singular SDEs, and \cite{S11} for
SDEs driven by a fractional Brownian motion. Following this line,
 in this paper we aim to establish transportation cost inequalities  for the
 segment processes associated to a class of neutral functional SDEs, which is unknown so far.
 The   point of our study is not the construction of the coupling as it is now more or less standard in the literature,
  but lies on the technical details to derive from the coupling reasonable estimates for which difficulties caused by the neutral part
  and functional coefficients have to be carefully managed.

Recall that a differential equation is  called neutral if, besides
the derivatives of the present state of the system, those of the
past history are also involved  (see \cite{M97}).  Let
$\C:=C([-\tau,0];\R^d)$ for some constant $\tau>0$, which is a
Banach space with the uniform norm $\|\cdot\|_\8 $. Let $\C$ be
equipped with the Borel $\si$-field induced by $\|\cdot\|_\8$. For
any $h\in C([-\tau,\8);\R^d)$ and $t\ge 0$, let $h_t\in\C$ such that
$h_t(\theta)= h(t+\theta), \theta\in [-\tau,0].$ We consider the
following neutral functional SDE on $\R^d$:
\begin{equation}\label{eq1}
\begin{cases}
\d\{X(t)-G(X_t)\}=b(X_t)\d t+\si(X_t)\d W(t), \ \ \ t\in[0,T],\\
X_0=\xi\in\mathscr{C},
\end{cases}
\end{equation}
where  $G,b:\C\to\R^d$ and $\si:\C\to\R^d\otimes\R^m$ are Lipschitz
continuous on bounded sets, and $W(\cdot)$ is an $\R^m$-valued
Brownian motion defined on a complete probability space
$(\OO,\F,\P)$ with the natural filtration $\{\F_{t}\}_{t\ge0}$.
Throughout this paper, we assume that for any initial data $X_0$,
  a $\C$-valued random variable independent of $W(\cdot),$ this equation has a unique global solution.
  This   can be ensured by   the strict contraction of $G$, i.e. $|G(\xi)-G(\eta)|\le \kk \|\xi-\eta\|_\8$ holds
  for some constant $\kk\in [0,1)$ and all $\xi,\eta\in\C$, together with the usual monotonicity and coercivity conditions
   of $b$ and $\si$, see e.g. \cite[Theorem 2.3]{Vs10}. We note that  the segment process $(X_t)_{t\ge 0}$ of the solution is a Markov process.

 As in \cite{W11}, we allow the initial data of the equation to be random, i.e. we consider the transportation cost inequality
 for the law of the solution starting from a probability measure $\mu$ on $\C$.
 In Section 2 we study the transportation cost inequality with respect to the  uniform distance on
  path space, while in Section 3 we consider the $L^2$ distance.
 Finally, in Section 4, we extend our results to a class of  neutral functional SPDEs.

\section{The uniform distance}
Let $T>0$ be fixed. For any $\xi\in \C$, let $\Pi^T_\xi$ be the
distribution of $X_{[0,T]}:=(X_t)_{t\in [0,T]}$ for the solution to
(\ref{eq1}) with $X_0=\xi$. Then,  for any $\mu\in\scr P(\C),$ the
set of all probability measures on $\C$, the distribution of
$X_{[0,T]}$ with initial distribution $\mu$ is given by
$$\Pi_\mu^T= \int_\C \Pi_\xi^T\mu(\d\xi).$$

For any probability density function $F$ of $\Pi_\mu^T$, i.e. $F$ is a non-negative measurable function on the free path space $C([0,T];\C)$ such that
 $\Pi_\mu^T(F):= \int_\C F\d\Pi_\mu^T=1,$ let $\mu_F^T$ be the marginal distribution of $F\Pi_\mu^T$ at time $0$. We have
$$\mu_F^T(\d\xi)= \Pi_\xi^T(F)\mu(\d\xi)\in \scr P(\C).$$ Let $\|\cdot\|$ and  $\|\cdot\|_{HS}$ denote the operator norm and  the Hilbert-Schmidt norm respectively.

To establish the transportation cost inequality for $\Pi_\mu^T$ with respect to the   uniform distance
\begin{equation}\label{eq011}
\rr_\8^T (\bar \xi,\bar \eta):=\sup_{t\in [0,T]}   \|\bar \xi_t-\bar
\eta_t\|_\8,\ \ \bar\xi,\bar\eta\in C([0,T];\C),
\end{equation}
we shall need the following conditions.

\beg{enumerate} \item[{\bf (A1)}]   There exists a constant $\kk\in [0,1)$ such that
   $$
|G(\xi)-G(\eta)|\le\kk\|\xi-\eta\|_\8,\ \ \xi,\eta\in\C.
$$
\item[{\bf (A2)}] There exist constants $\ll_1\in\R$ and $ \ll_2\ge 0$ such that
\begin{equation*}
\begin{split}
& 2\big\< \xi(0)-\eta(0) - G(\xi)+G(\eta),
b(\xi)-b(\eta)\big\>+\|\si(\xi)-\si(\eta)\|_{HS}^2
 \le \ll_1\|\xi-\eta\|^2_\8,\\
&\|\si(\xi)-\si(\eta)\|_{HS}^2\le\ll_2\|\xi-\eta\|_\8^2, \ \
\xi,\eta\in\C.
\end{split}
\end{equation*}\item[{\bf (A3)}] There exists a constant  $\ll_3> 0$ such that $\|\si(\xi)\|\le\ll_3$ for all $\xi\in\C.$
\end{enumerate}

\

Let $\ll_1^+=0\lor\ll_1$ and $\ll_1^-= 0\lor (-\ll_1).$ We will need
the following two quantities:
\begin{equation}\label{eq014}
\alpha(T):=\ff{2\ll_3(1+\kk)^2}{(1-\kk)^2}\min\bigg\{
\ff{(4\ss{\ll_2}+\ss{16\ll_2+\ll_1^+})^2}{(\ll_1^+)^2},\
\ff{4T\exp\big[1+\ff{2\ll_1^-+4\ll_2}{(1-\kk)^2}T\big]}{2T\ll_1^++(1-\kk)^2}\bigg\}
\end{equation}
\begin{equation}\label{eq015}
\bb(T):=1+\ff{(1+\kk)^2}{(1-\kk)^2}
\min\bigg\{\ff{(2\ss{\ll_2}+\ss{4\ll_2+\ll_1^+})^2}{\ll_1^+},\
2\exp\Big[\ff{2\ll_1^-+16\ll_2}{(1-\kk)^2}T\Big]\bigg\}.
\end{equation}

The main result of this section is the following.

\begin{thm}\label{Theorem 1} Assume {\bf (A1)}-{\bf (A3)} and let
\begin{equation}\label{eq013}
\rr(\xi,\eta):=\|\xi-\eta\|_\8, \ \ \ \xi,\eta\in\C.
\end{equation}
For any $T>0,\mu\in\scr P(\C)$ and non-negative measurable function
$F$ on $C([0,T];\C)$ such that $\Pi_\mu^T(F)=1,$
 \begin{equation}\label{eq16}
W_{2,\rr_\8^T}(F\Pi^T_\mu,\Pi^T_\mu)\le\ss{\bb(T)}W_{2,\rr}(\mu,\mu^T_F)+\ss{\aa(T)}\ss{\Pi_\mu^T(F\log
F)}.
\end{equation} If moreover $\mu$ satisfies the
transportation cost inequality
\begin{equation}\label{eq29}
W_{2,\rr}(\mu,f\mu)^2\le c_\mu \mu(f\log f),\ \ f\ge 0,\mu(f)=1
\end{equation}
for some constant $c_\mu>0$, then
\begin{equation}\label{eq30}
W_{2,\rr_\8^T}(F\Pi^T_\mu,\Pi^T_\mu)^2\le\Big(\ss{\aa(T)}+\ss{c_\mu
\bb(T)}\Big)^2\Pi_\mu^T(F\log F).
\end{equation}
\end{thm}

\begin{proof} The proof is based on the following   Lemma \ref{Lemma 1} and Lemma
\ref{Lemma 2}.
By the triangle inequality it follows that
\begin{equation*}
W_{2,\rr_\8^T}(F\Pi^T_\mu,\Pi^T_\mu)\le
W_{2,\rr_\8^T}(F\Pi^T_\mu,\Pi^T_{\mu^T_F})+W_{2,\rr_\8^T}(\Pi^T_\mu,\Pi^T_{\mu^T_F}).
\end{equation*}
Then \eqref{eq16} follows from Lemma \ref{Lemma 1} and Lemma
\ref{Lemma 2}, and   \eqref{eq30} is a direct consequence of  \eqref{eq16} and
\eqref{eq29}.
\end{proof}

Let $\mu=\dd_\xi$ for $\xi\in\C$. Then (\ref{eq29}) holds for $c_\mu=0$, so that (\ref{eq30}) becomes
$$W_{2,\rr_\8^T}(F\Pi^T_\xi,\Pi^T_\xi)^2\le \aa(T) \Pi_\xi^T(F\log
F).$$ This inequality also follows from the following lemma since in this case we have $\mu=\mu_F^T=\dd_\xi$.

\begin{lem}\label{Lemma 1} Assume {\bf (A1)}-{\bf (A3)}.
  For any $\mu\in\scr P(\C)$ and $T>0$,
\begin{equation}\label{eq26}
W_{2,\rr^T_\8}(F\Pi^T_\mu,\Pi^T_{\mu^T_F})^2\le\aa(T)\Pi_\mu^T(F\log
F),\ \ F\ge 0, \ \Pi_\mu^T(F)=1.
\end{equation}
\end{lem}

\begin{proof}
The main idea of the  proof is taken from    \cite[Proof of Theorem
1.1]{W11}, which indeed goes back to \cite{DGW04}. According to (b)
in the proof of \cite[Theorem 1.1]{W11}, we may  and do assume that
$\mu=\dd_\xi, \xi\in\C$. In this case
$\Pi_\mu^T=\Pi_{\mu^T_F}^T=\Pi^T_\xi$. For a positive bounded
measurable function $F$ on $C([0,T];\C)$ such that $ \Pi^T_\xi(F)=1$
and $\inf F>0$, define
\begin{equation*}
m(t):=\E(F(X_{[0,T]})|\F_t) \ \ \mbox{ and }\ \ L(t):=\int_0^t\ff{\d
m(s)}{m(s)},\ \ \ t\in[0,T],
\end{equation*}
where $\E$ is the expectation taken for the probability measure
$\P$. Then $m(t)$ and $L(t)$ are square-integrable
$\F_t$-martingales under $\P$ due to  $\inf F>0$ and the boundedness
of $F$. Note by the It\^o formula that
\begin{equation}\label{eq2}
m(t)=e^{L(t)-\ff{1}{2}\<L\>(t)},
\end{equation}
where $\<L\>(t)$ denotes the quadratic variation process of $L(t)$,
and, by the martingale representation theorem, e.g., \cite[Theorem
6.6]{IW89}, there exists a unique $\R^m$-valued
$\F_t$-predictable process $h(t)$ such that
\begin{equation}\label{eq3}
L(t)=\int_0^t\< h(s),\d W(s)\>.
\end{equation}
 Since $F(X_{[0,T]})$ is $\F_T$-measurable and $\< L\>(t)=\int_0^t|h(s)|^2\d s$, it
then follows from \eqref{eq2} and \eqref{eq3} that
\begin{equation*}
F(X_{[0,T]})=m(T)=\exp\bigg[\int_0^T\< h(s),\d
W(s)\>-\int_0^T|h(s)|^2\d s\bigg].
\end{equation*}
Let
\begin{equation}\label{eq05}
\d\Q=F(X_{[0,T]})\d\P.
\end{equation}
Then $\Q$ is a probability measure on $\OO$ due to $\Pi_\xi^T(F)=1$.
To prove the desired inequality, we need to characterize
$\Pi_\xi^T(F\log F)$ and $W_{2,\rr_\8^T}(F\Pi_\xi^T,\Pi_\xi^T)$
respectively.

\

(i) Recalling that $F(X_{[0,T]})=m(T),$   $m(t)$ is a
square-integrable $\F_t$-martingale under $\P$, and observing that
$h(s)$ is $\F_s$-measurable, we have
\begin{equation}\label{eq10}
\begin{split}
\E_\Q|h(s)|^2=\E(m(T)|h(s)|^2)=\E(|h(s)|^2\E(m(T)|\F_s))=\E(|h(s)|^2m(s)).
\end{split}
\end{equation}
Moreover, by the It\^o formula
\begin{equation}\label{eq8}
\begin{split}
\d(m(s)\log m(s))&=(1+\log m(s))\d m(s)+\ff{\d\<
m\>(t)}{2m(t)}\\
&=(1+\log m(s))\d m(s)+\ff{m(s)}{2}|h(s)|^2\d t,
\end{split}
\end{equation}
where we have used the fact that
\begin{equation*}
\d \< m\>(s)=m^2(s)\d \< L\>(s)=m^2(s)|h(s)|^2\d s.
\end{equation*}
  Since $m(t)$ is a
square-integrable $\F_t$-martingale under $\P$, integrating from $0$
to $T$ and taking expectations with respect to $\P$ on both sides of
\eqref{eq8}, we get
\begin{equation}\label{eq11}\Pi_\xi^T(F\log F)=
\E(m(T)\log m(T))=\ff{1}{2}\int_0^T\E(m(t)|h(t)|^2)\d t=\ff 1 2
\int_0^T \E_\Q|h(t)|^2\d t.
\end{equation}

(ii) Recalling that $m(t)$ is a  square-integrable $\F_t$-martingale
under $\P$, we deduce from the Girsanov theorem that
\begin{equation}\label{eq06}
\tt{W}(t):=W(t)-\int_0^th(s)\d s
\end{equation}
is an $m$-dimensional $\F_t$-Brownian motion on the probability
space $(\Omega,\F,\Q)$.
 Reformulate (\ref{eq1}) as
\begin{equation}\label{eq34}
\begin{cases}
\d\{X(t)-G(X_t)\}=\{b(X_t)+\si(X_t)h(t)\}\d t+\si(X_t)\d \tt{W}(t),\ \ \ t\in[0,T],\\
X_0=\xi.
\end{cases}
\end{equation}
Noting that the law of  $X_{[0,T]}$ under $\P$ is $\Pi^T_\xi$  and
$\d\Q=F(X_{[0,T]})\d\P$, for any bounded measurable function $G$ on
$C([0,T];\C)$, we have
\begin{equation*}
\E_\Q(G(X_{[0,T]}))=\E(FG)(X_{[0,T]})=\Pi^T_\xi(FG).
\end{equation*}
Hence the law of $X_{[0,T]}$ under $\Q$ is $F\Pi^T_\xi$. Next,
consider  the following equation
\begin{equation}\label{eq35}
\begin{cases}
\d\{Y(t)-G(Y_t)\}=b(Y_t)\d t+\si(Y_t)\d \tt{W}(t),\ \ \ t\in[0,T],\\
Y_0=\xi.
\end{cases}
\end{equation}
Since $\tt W(t)$ is the Brownian motion  under $\Q$, we conclude
that the law of  $ Y_{[0,T]} $ under $\Q$ is $\Pi^T_\xi$. This,
together with $X_0=Y_0$ and the law of $ X_{[0,T]} $ under $\Q$ is
$F\Pi^T_\xi$, leads to
\begin{equation}\label{eq6}
W_{2,\rr_\8^T}(F\Pi^T_\xi,\Pi^T_\xi)^2\le\E_\Q\rr_\8^T(X_{[0,T]},Y_{[0,T]})^2
=\E_\Q\Big(\sup_{0\le t\le T}|X(t)-Y(t)|^2\Big).
\end{equation}

\

Now, combining (\ref{eq6}) with (\ref{eq11}), we need only to prove
the inequality
\begin{equation}\label{eq63}
\E_\Q\Big(\sup_{0\le t\le
T}|X(t)-Y(t)|^2\Big)\le\ff{\aa(T)}2\int_0^T\E_\Q|h(t)|^2\d t.
\end{equation}
 Let
$ M(t)=(X(t)-Y(t))+(G(Y_t)-G(X_t)).$ By {\bf (A1)} and the inequality
\begin{equation}\label{eq21}
(a+b)^2\le(1+\ee) (a^2+b^2/\ee),\ \ee>0,
\end{equation}
 we obtain that
\begin{equation}\label{eq57}
\begin{split}
|M(s)|^2&\le(1+\kk)(|X(s)-Y(s)|^2+|G(Y_s)-G(X_s)|^2/\kk)\\
&\le(1+\kk)^2\|X_s-Y_s\|^2_\8,
\end{split}
\end{equation}
and
\begin{equation}\label{eq68}
\begin{split}
|X(s)-Y(s)|^2&=|M(s)+(G(X_s)-G(Y_s))|^2\\
&\le\kk\|X_s-Y_s\|_\8^2+\ff{1}{1-\kk}|M(s)|^2.
\end{split}
\end{equation}
It thus follows from $X_0=Y_0$ that
\begin{equation}\label{eq58}
(1-\kk)^2\sup_{0\le s\le t}|X(s)-Y(s)|^2\le\sup_{0\le s\le
t}|M(s)|^2\le(1+\kk)^2\sup_{0\le s\le t}|X(s)-Y(s)|^2.
\end{equation}
 By {\bf (A2)}, {\bf (A3)} and
It\^o's formula, one has
\begin{equation}\label{eq12}
\begin{split}
\d |M(t)|^2&\le2\< M(t),(\si(X_t)-\si(Y_t))\d
\tt{W}(t)\>\\
&\quad+\Big(2\ss{\ll_3}\,|M(t)|\cdot
|h(t)|-\ll_1\|X_t-Y_t\|_\8^2\Big)\d t,
\end{split}
\end{equation}
which, together with the inequality $2ab\le \dd a^2+b^2/\dd, \
\dd>0$, and \eqref{eq57}, gives that
\begin{equation}\label{eq61}
\begin{split}
\d |M(t)|^2 &\le2\< M(t),(\si(X_t)-\si(Y_t)\d
\tt{W}(t))\>\\
&\quad+\Big(\ff{\ll_3}{\dd}(1+\kk)^2|h(t)|^2+(\dd-\ll_1)\|X_t-Y_t\|_\8^2\Big)\d
t,\ \ \ \dd>0.
\end{split}
\end{equation}
Due to the Burkhold-Davis-Gundy inequality and {\bf(A2)}, this implies that
\begin{equation*}
\begin{split}
\E_\Q\Big(\sup_{0\le s\le
t}|M(s)|^2\Big)&\le4\ss{\ll_2}\E_\Q\Big(\int_0^t|M(s)|^2\|X_s-Y_s\|_\8^2\d
s\Big)^{\ff{1}{2}}\\
&\quad+(\dd-\ll_1)^+\E_\Q\int_0^t\|X_s-Y_s\|_\8^2 \d
s+\ff{\ll_3}{\dd}(1+\kk)^2\int_0^t\E_\Q|h(s)|^2\d
s\\
 &\le\Big((\dd-\ll_1)^++\ff{4\ll_2}\vv\Big)\E_\Q\int_0^t\|X_s-Y_s\|_\8^2 \d
s +\vv\E_\Q\Big(\sup_{0\le s\le
t}|M(s)|^2\Big)\\
&\quad +\ff{\ll_3(1+\kk)^2}{\dd}\int_0^t\E_\Q|h(s)|^2\d s,\ \
\dd>0,\vv\in (0,1).
\end{split}
\end{equation*} By an approximation argument using stopping times, we may assume that $ \E_\Q\Big(\sup_{0\le s\le
t}|M(s)|^2\Big)<\infty$, so that this is equivalent to
\begin{equation*}
\begin{split}
\E_\Q\Big(\sup_{0\le s\le t}|M(s)|^2\Big)
 &\le \Big(\ff{(\dd-\ll_1)^+}{1-\vv}+\ff{4\ll_2}{\vv(1-\vv)}\Big)  \E_\Q\int_0^t\|X_s-Y_s\|_\8^2 \d
s\\
&\quad+\ff{\ll_3(1+\kk)^2}{\dd (1-\vv)}\int_0^t\E_\Q|h(s)|^2\d s,\ \
\dd>0,\vv\in (0,1).
\end{split}
\end{equation*}
Thus, \eqref{eq58} yields that
\begin{equation}\label{eq70}
\begin{split}
\E_\Q\Big(\sup_{0\le s\le t}|X(s)-Y(s)|^2\Big)
 &\le\ff{\vv(\dd-\ll_1)^++4\ll_2}{\vv(1-\vv)(1-\kk)^2}\E_\Q\int_0^t\|X_s-Y_s\|_\8^2 \d
s\\
&\quad+\ff{\ll_3(1+\kk)^2}{\dd(1-\kk)^2(1-\vv)}\int_0^t\E_\Q|h(s)|^2\d
s,\ \ \dd>0,\vv\in (0,1).
\end{split}
\end{equation}
Then, by the Gronwall inequality,
$$
\E_\Q\Big(\sup_{0\le t\le T}|X(t)-Y(t)|^2\Big)
  \le\ff{ \ll_3(1+\kk)^2\exp\big[\ff{\vv(\dd-\ll_1)^++4\ll_2}{\vv(1-\vv)(1-\kk)^2}T\big]  }{\dd (1-\kk)^2(1-\vv)}
 \int_0^T\E_\Q|h(t)|^2\d t$$ holds for all $\dd>0$ and $\vv\in (0,1).$  Taking $\vv =\ff 1 2$ and $\dd= \ll_1^++\ff{(1-\kk)^2}{2T},$ we obtain
 \begin{equation}\label{eq64}
 \E_\Q\Big(\sup_{0\le t\le T}|X(t)-Y(t)|^2\Big)
 \le\ff{4\ll_3(1+\kk)^2T\exp\big[1+\ff{2\ll_1^-+16\ll_2}{(1-\kk)^2}T\big]}{(1-\kk)^2\{2T\ll_1^++(1-\kk)^2\}}  \int_0^T\E_\Q|h(t)|^2\d t.
\end{equation}

On the other hand, if $\ll_1>0$,  taking $\dd=\lambda_1/2$ in
\eqref{eq61} we obtain
\begin{equation}\label{eq60}
\E_\Q\int_0^t\|X_s-Y_s\|_\8^2\d
s\le\ff{4\ll_3(1+\kk)^2}{\ll_1^2}\int_0^t\E_\Q|h(s)|^2\d s.
\end{equation}
Combining this with (\ref{eq70}) with $\dd=\ll_1$ we derive
$$\E_\Q\Big(\sup_{0\le t\le T}|X(t)-Y(t)|^2\Big)\le \ff{\ll_3(1+\kk)^2}{\ll_1(1-\kk)^2}
\Big(\ff{16\ll_2}{\vv(1-\vv)\ll_1}+\ff 1 {1-\vv}\Big)\int_0^T \E_\Q|h(s)|^2\d
s.$$ Taking the optimal choice
$$\vv=\ff{4\ss{\ll_2}}{4\ss{\ll_2}+\ss{16\ll_2+\ll_1}},$$ we conclude that
$$\E_\Q\Big(\sup_{0\le t\le T}|X(t)-Y(t)|^2\Big)\le \ff{\ll_3(1+\kk)^2(4\ss{\ll_2}+\ss{16\ll_2+\ll_1})^2}{\ll_1^2(1-\kk)^2} \int_0^T \E_\Q|h(s)|^2\d
s.$$ Combining this with (\ref{eq64}) we prove (\ref{eq63}), and hence, finish the proof.
\end{proof}

\begin{lem}\label{Lemma 2}
  Let {\bf(A1)} and {\bf (A2)} hold. Then
 \begin{equation}\label{eq14}
W_{2,\rr_\8^T}(\Pi^T_\nu,\Pi^T_\mu)^2\le\bb(T)W_{2,\rr}(\nu,\mu)^2,\
\ \mu,\nu\in\scr P(\C).
\end{equation}
 \end{lem}

\begin{proof}
Let $\{X(t)\}_{t\ge 0},\{Y(t)\}_{t\ge 0}$ be the solutions to
\eqref{eq1} with   $X_0=\xi$ and $Y_0=\eta$, where $\xi$ and $\eta$ are $\C$-valued random variables   with distributions $\mu$ and $\nu$ respectively
 and are independent of $W(\cdot)$ such that
 \begin{equation}\label{eq13}
\E(\|\xi-\eta\|^2_\8)=W_{2,\rr} (\nu,\mu)^2.
\end{equation}
Then it suffices to show that
\begin{equation}\label{eq67}
\E\Big(\sup_{t\in [0,T]} \|X_t-Y_t\|_\8^2\Big)
\le\bb(T)\E(\|\xi-\eta\|_\8^2).
\end{equation}
Let $h=0$. We have $\tt W=W$ so that    \eqref{eq61}
still holds for $W$ in place of  $\tt{W}$. Combining it with
\eqref{eq57}, we obtain that when $\ll_1>0,$
\begin{equation}\label{eq66}
\E\int_0^t  \|X_s-Y_s\|_\8^2\d s\le\ff{1}{ \ll_1
}\E|M(0)|^2\le\ff{(1+\kk)^2}{ \ll_1}\E\|\xi-\eta\|^2_\8.
\end{equation} Similarly, since in the present case  $h=0$ and according to (\ref{eq57}), $|M(0)|^2\le (1+\kk)^2\|\xi-\eta\|^2_\8,$ we may
take $\dd=0$ in the argument leading to  (\ref{eq70})  to derive   that
\begin{equation}\label{OPQ}
\begin{split}
&\E \Big(\sup_{0 \le s\le t} |X(s)-Y(s)|^2\Big)\\
 &\le \ff{\vv\ll_1^-+4\ll_2 }{\vv(1-\vv)(1-\kk)^2}\E \int_0^t  \|X_s-Y_s\|_\8^2 \d s   +\ff{(1+\kk)^2}{(1-\vv)(1-\kk)^2}\E\|\xi -\eta \|^2
\end{split}
\end{equation} for $\vv\in (0,1).$
When $\ll_1>0$, combining this with (\ref{eq66}) we arrive at
\beg{equation*}\beg{split}&\E\Big(\sup_{t\in [0,T]} \|X_t-Y_t\|_\8^2\Big) \le \E \Big(\sup_{s\in [0,T]} |X(s)-Y(s)|^2\Big) +\E\|\xi-\eta\|_\8^2\\
&\le
  \bigg\{1+\ff {(1+\kk)^2} {(1-\kk)^2}\Big(\ff 1 {1-\vv} +\ff{4\ll_2}{\vv(1-\vv)\ll_1}\Big)\bigg\}\E\|\xi-\eta\|_\8^2.\end{split}
\end{equation*} Taking $$\vv= \ff{2\ss{\ll_2}}{2\ss{\ll_2}+\ss{4\ll_2+\ll_1}}$$ we deduce that
\beq\label{OOP} \beg{split} &\E\Big(\sup_{t\in [0,T]} \|X_t-Y_t\|_\8^2\Big)\\
&\le
\bigg(1+\ff{(1+\kk)^2(2\ss{\ll_2}+\ss{4\ll_2+\ll_1})^2}{\ll_1(1-\kk)^2}\bigg)\E\|\xi-\eta\|_\8^2,\
\ \ll_1>0.\end{split} \end{equation} In general,  by the Gronwall inequality,
(\ref{OPQ}) yields that
\beg{equation*}\beg{split}&\E\Big(\sup_{t\in [0,T]} \|X_t-Y_t\|_\8^2\Big) \le \E \Big(\sup_{s\in [0,T]} |X(s)-Y(s)|^2\Big) +\E\|\xi-\eta\|_\8^2\\
&\le \bigg(1+
\ff{(1+\kk)^2}{(1-\vv)(1-\kk)^2}\exp\Big[\ff{\vv\ll_1^-+4\ll_2}{\vv(1-\vv)(1-\kk)^2}T\Big]\bigg)
\E\|\xi-\eta\|_\8^2.\end{split}\end{equation*} Taking $\vv=\ff 1 2$
we obtain
$$\E\Big(\sup_{t\in [0,T]} \|X_t-Y_t\|_\8^2\Big) \le \bigg(1+ \ff{2(1+\kk)^2}{(1-\kk)^2}\exp\Big[\ff{2\ll_1^-+16\ll_2}{(1-\kk)^2}T\Big]\bigg) \E\|\xi-\eta\|_\8^2.$$
Combining this with (\ref{OOP}) we prove (\ref{eq67}), and hence,
finish the proof.\end{proof}

\paragraph{Remark 2.1.} Obviously, when $\ll_1>0$   both $\aa(T)$ and $\bb(T)$ are bounded in $T$, so that Theorem \ref{Theorem 1} works also for $T=\8$, i.e. on the global free path space $C([0,\8);\C).$ Precisely, let $\Pi_\mu$ and $\Pi_\xi$ denote the distribution of $X_{[0,\8)}$ with initial distributions $\mu$ and $\dd_\xi$ respectively, let $\mu_F(\d\xi)= \Pi_\xi(F) \mu(\d\xi),$ and let $$\rr_\8(\bar\xi,\bar\eta)=\sup_{t\ge 0}\rr_\8(\bar\xi_t,\bar\eta_t),\ \ \bar\xi,\bar\eta\in C([0,\8);\C).$$ If $\ll_1>0$, then Theorem \ref{Theorem 1} implies
\beg{equation*}\beg{split} W_{2,\rr_\8}(F\Pi_\mu,\Pi_\mu)\le &\ff{
\ss{2\ll_3}(1+\kk)(4\ss{\ll_2}+\ss{16\ll_2+\ll_1})}{(1-\kk)\ll_1}
\ss{\Pi_\mu(F\log F)} \\
&+\bigg(1+\ff{(1+k)(2\ss{\ll_2}+\ss{4\ll_2+\ll_1})}{(1-\kk)\ss{\ll_1}}\bigg)W_{2,\rr}(\mu,\mu_F).\end{split}\end{equation*}

In general, for any $\ll_1\in\R$, we can find $\ll>0$ and constants
$C_1(\ll), C_2(\ll)>0$ such that \beq\label{GG0}
W_{2,\rr_{\8,\ll}}(F\Pi_\mu,\Pi_\mu)\le  C_1(\ll) \ss{\Pi_\mu(F\log
F)}  +C_2(\ll)W_{2,\rr}(\mu,\mu_F),\end{equation} where
$$\rr_{\8,\ll}(\bar\xi,\bar\eta):=\sup_{t\ge 0}\big\{\e^{-\ll t} \rr_\8(\bar\xi_t,\bar\eta_t)\big\},\ \ \bar\xi,\bar\eta\in C([0,\8);\C).$$
Indeed, for any $\ll>\ff{ \ll_1^-+8\ll_2}{(1-k)^2}$, \beq\label{GG1}
\sum_{n=1}^\8 \e^{-2\ll
n}\big\{\aa(n)+\bb(n)\big\}<\8.\end{equation} Noting that
$$\rr_{\8,\ll}(\bar\xi,\bar\eta)^2\le \sum_{n=1}^\8 \e^{-2\ll (n-1)} \rr_\8^n(\bar\xi_{[0,n]},\bar\eta_{[0,n]})^2,$$ we have
$$W_{2,\rr_{\8,\ll}}^2\le \sum_{n=1}^\8\e^{-2\ll(n-1)} W_{2,\rr_\8^n}^2.$$ Combining this with Theorem \ref{Theorem 1} and (\ref{GG1}), we may find finite constants $C_1(\ll),C_2(\ll)>0$ such that (\ref{GG0}) holds.

\section{The weighted $L^2$ distance on $C([0,\8);\C)$}

Since for a fixed $T>0$ the $L^2$-distance on $C([0,T];\C)$ is
dominated by the uniform norm,   the corresponding
transportation cost inequality is weaker than that derived in
Section 2. So,  in this section we only consider the global path space
$C([0,\8);\C)$. Let \beg{equation}\label{eq07} \rr_2(\xi,\eta)^2=
\ff 1 {\tau}\int_{-\tau}^0|\xi(\theta)-\eta(\theta)|^2\d\theta,\ \
\xi,\eta\in\C,\end{equation} and for $\ll\ge 0$ let
\beg{equation}\label{eq08} \rr_{2,\ll}(\bar\xi,\bar\eta)^2=
\int_0^\8\e^{-\ll t} \rr_2(\bar\xi_t,\bar \eta_t)^2\d t,\ \
\bar\xi,\bar\eta\in C([0,\8);\C).\end{equation} As mentioned in
Remark 2.1, let $\Pi_\mu$ and $\Pi_\xi$ denote the distribution of
$X_{[0,\8)}$ with initial distributions $\mu$ and $\dd_\xi$
respectively. Let $\mu_F(\d\xi)= \Pi_\xi(F) \mu(\d \xi).$

To derive the transportation cost inequality w.r.t. $\rr_{2,\ll}$, we need the following assumptions to replace {\bf (A1)} and {\bf (A2)} in the last section.

\beg{enumerate}\item[{\bf (B1)}] There exists a constant $k\in
[0,1)$ such that $|G(\xi)-G(\eta)|\le k \rr_2(\xi,\eta), \
\xi,\eta\in\C.$
\item[{\bf (B2)}] There exist constants $k_1\in\R,k_2\ge 0$  and a probability measure $\LL$ on $[-\tau,0]$ such that
\begin{equation*}
\begin{split}
&2\big\<(\xi(0)-\eta(0))-G(\xi)+G(\eta),b(\xi)-b(\eta)\big\>+\|\si(\xi)-\si(\eta)\|_{HS}^2\\
&\le-k_1 |\xi(0)-\eta(0)|^2+k_2 \int_{-\tau}^0
|\xi(\theta)-\eta(\theta)|^2\LL(\d\theta).
\end{split}
\end{equation*}
\end{enumerate}

\

A simple  example  such that
{\bf (B1)} and  {\bf (B2)} hold is that
\beg{equation*}\beg{split} &G(\xi)= \ff k {\tau}\int_{-\tau}^0 \xi(\theta)\d\theta,\\
&b(\xi)= c_1\xi(0) + \int_{-\tau}^0\xi(\theta)\LL_1(\d\theta),\\
&\si(\xi)= c_3\xi(0)+ \int_{-\tau}\xi(\theta) \LL_2(\d\theta)\end{split}\end{equation*} for some constants $k\in (0,1), c_1 \in\R$ and some finite measures $\LL_1,\LL_2$ on $[-\tau,0]$.

 \beg{thm}\label{T3.1}  Assume {\bf (B1)}, {\bf(B2)} and {\bf (A3)}.  Let $\tt\rr_2(\xi,\eta)^2=|\xi(0)-\eta(0)|^2+\rr_2(\xi,\eta)^2, \xi,\eta\in\C.$  Let $\mu\in\scr P(\C)$ and $F$ be  non-negative measurable function $F$ on $C([0,\8);\C)$ such that $\Pi_\mu(F)=1.$
\beg{enumerate} \item[$(1)$]  If $k_1>k_2$ then
\beg{equation*}\beg{split} W_{2,\rr_{2,0}}(\Pi_\mu,F\Pi_\mu)\le &\ff{\ss{2\ll_3}\{1+(1+k)^2\}}{k_1-k_2}\ss{\Pi_\mu(F\log F)}\\
 &+\ss{\tau +\ff{k_2\tau+1+k}{k_1-k_2}}\,W_{2,\tt\rr_2}(\mu,\mu_F).\end{split}\end{equation*}
\item[$(2)$] If $k_1\le k_2$ then for any $\ll> \ff{k_2-k_1}{(1-k)^2},$
\beg{equation*}\beg{split}W_{2,\rr_{2,\ll}}(\Pi_\mu,F\Pi_\mu)\le &\ff{\ss{2\ll_3}\{1+(1+k)^2\}}{k_1-k_2+\ll (1-k)^2}\ss{\Pi_\mu(F\log F)}\\
 &+\ss{\tau +\ff{\ll k(1-k)\tau +k_2\tau+1+k}{\ll(1-k)^2+k_1-k_2}}\,W_{2,\tt\rr_2}(\mu,\mu_F).\end{split}\end{equation*}\end{enumerate} \end{thm}
As explained in the proof of Theorem \ref{Theorem 1} that the result follows immediately from Lemmas \ref{L3.2} and \ref{L3.3} below. To prove these lemmas, we first collect some simple facts.

\beg{lem}\label{L3.0} Assume {\bf (B1)}. Let $t>0, \ll\ge 0, \bar\xi,\bar\eta\in C([0,t];\C),$ and $\LL$ be a probability measure on $[-\tau,0]$. Let
$$\bar M(s)= \bar\xi(s)-\bar\eta(s)-G(\bar\xi_s)+G(\bar\eta_s).$$ Then \beg{enumerate}\item[$(1)$] $\int_0^t \e^{-\ll s} \d s\int_{-\tau}^0|\bar \xi(s+\theta)-\bar\eta(s+\theta)|^2   \LL(\d\theta) \le \tau\rr_2(\bar\xi_0,\bar\eta_0)^2 + \int_0^t \e^{-\ll s} | \bar\xi (s)- \bar\eta(s)|^2\d s.$
\item[$(2)$] $\int_0^t\e^{-\ll s} |\bar M(s)|^2\d s\le (1+k)^2 \int_0^t \e^{-\ll s}|\bar \xi(s)-\bar\eta(s)|^2\d s +(1+k)k\tau\rr_2(\bar\xi_0,\bar\eta_0)^2.$
\item[$(3)$] $\int_0^t\e^{-\ll s}|\bar \xi(s)-\bar\eta(s)|^2\d s\le \ff 1 {(1-k)^2} \int_0^t \e^{-\ll s} |\bar M(s)|^2\d s+\ff{k\tau}{1-k}\rr_2(\bar\xi_0,\bar\eta_0)^2.$
 \end{enumerate} \end{lem}

\beg{proof} (1) By the Fubini theorem, We have
\beg{equation*}\beg{split} &\int_0^t \e^{-\ll s} \d s\int_{-\tau}^0|\bar \xi(s+\theta)-\bar\eta(s+\theta)|^2   \LL(\d\theta)\\
&=\int_{-\tau}^0\LL(\d\theta) \int_\theta^{t+\theta} \e^{-\ll (s-\theta)}|\bar\xi(s)-\bar\eta(s)|^2\d s\\
&\le \int_0^t \e^{-\ll s}|\bar\xi(s)-\bar\eta(s)|^2\d s +
\tau\rr_2(\bar\xi_0,\bar\eta_0)^2.\end{split}\end{equation*}

(2)   By {\bf (B1)} and applying (\ref{eq21}) to $\vv=k$, we obtain
\beq\label{WFF}|\bar M(s)|^2 \le (1+k) \big\{|\bar\xi(s)-\bar\eta(s)|^2 +k \rr_2(\bar\xi_s,\bar\eta_s)^2\big\}.\end{equation} Then
$$\int_0^t|\bar M(s)|^2\e^{-\ll s}\d s \le (1+k)\int_0^t\big\{|\bar\xi(s)-\bar\eta(s)|^2 +k \rr_2(\bar\xi_s,\bar\eta_s)^2\big\}\e^{-\ll s}\d s.$$
On the other hand, taking $\LL(\d \theta)= \ff 1 \tau \d\theta$ on $[-\tau,0]$, we have
\beq\label{X}\int_0^t\rr_2(\bar\xi_s,\bar\eta_s)^2\e^{-\ll s}\d s\le \int_0^t |\bar\xi(s)-\bar\eta(s)|^2\e^{-\ll s}\d s +\tau\rr_2(\bar\xi_0,\bar\eta_0)^2.\end{equation} Therefore, the second assertion follows.

(3) By {\bf (B1)} and (\ref{eq21}) with $\vv= \ff k{1-k},$ we have
$$|\bar\xi(s)-\bar\eta(s)|^2 \le k \rr_2(\bar\xi_s,\bar\eta_s)^2 +\ff 1 {1-k} |\bar M(s)|^2.$$ Combining this with (\ref{X}) we arrive at
\beg{equation*}\beg{split} & \int_0^t |\bar\xi(s)-\bar\eta(s)|^2\e^{-\ll s}\d s \\
 &\le k \int_0^t |\bar\xi(s)-\bar\eta(s)|^2\e^{-\ll s}\d s  + k\tau \rr_2(\bar\xi_0,\bar\eta_0)^2 +\ff 1 {1-k} \int_0^t |\bar M(s)|^2\e^{-\ll s}\d s.\end{split}\end{equation*} This implies the third assertion.
\end{proof}

\beg{lem}\label{L3.2}  Assume {\bf (B1)}, {\bf(B2)} and {\bf (A3)}. \beg{enumerate} \item[$(1)$]  If $k_1>k_2$ then
$$W_{2,\rr_{2,0}}(F\Pi_\mu,\Pi_{\mu_F})^2\le \ff{2\ll_3\{1+(1+k)^2\}^2}{(k_1-k_2)^2}\Pi_\mu(F\log
F),\ \ F\ge 0, \Pi_\mu(F)=1.$$
\item[$(2)$] If $k_1\le k_2$ then for any $\ll> \ff{k_2-k_1}{(1-k)^2},$
$$W_{2,\rr_{2,\ll}}(F\Pi_\mu,\Pi^T_{\mu_F})^2\le \ff{2\ll_3\{1+(1+k)^2\}^2}{\{k_1-k_2+\ll (1-k)^2\}^2}\Pi_\mu(F\log
F),\ \ F\ge 0, \Pi_\mu(F)=1.$$\end{enumerate} \end{lem}

\beg{proof} By an approximation argument, it suffices to prove the result for $\Pi_\mu^T$ and  $\rr_{2,\ll}^T$ in place of $\Pi_\mu$ and $\rr_{2,\ll}$ respectively with arbitrary $T>0,$  where
$$\rr_{2,\ll}^T(\bar\xi,\bar\eta)^2:= \int_0^T \e^{-\ll t} \rr_2(\bar\xi_t,\bar\eta_t)^2\d t,\ \ \bar\xi,\bar\eta\in C([0,T];\C). $$   As indicated in the proof of Lemma \ref{Lemma 1}
 that we may and do assume $\mu=\dd_\xi$. Let $h, \tt W(t), \Q,$ $X(t),Y(t)$  and $M(t)$ be constructed in the proof of Lemma \ref{Lemma 1}. It suffices to prove that
 \beq\label{WWF0} \E_\Q\int_0^T \e^{-\ll t} \rr_2(X_t,Y_t)^2\d t\le C(\ll) \E_Q\int_0^T|h(t)|^2\d t\end{equation} for
 $$C(\ll)= \beg{cases} \dfrac{\ll_3\{1+(1+k)^2\}^2}{(k_1-k_2)^2}, &\text{if}\ k_1>k_2, \ll=0,\\
 \dfrac{\ll_3\{1+(1+k)^2\}^2}{\{k_1-k_2+\ll (1-k)^2\}^2}, &\text{if}\ k_1\le k_2, \ll> \ff{k_2-k_1}{(1-k)^2}.\end{cases}$$
 By {\bf (B2}), {\bf (A3}) and It\^o's formula, we obtain
 \beg{equation*}\beg{split}  &\d|M(t)|^2- 2\<M(t), \{\si(X_t)-\si(Y_t)\}\d\tt W(t)\> \\
 &\le  \bigg\{k_2\int_{-\tau}^0 |X(t+\theta)-Y(t+\theta)|^2\LL(\d\theta)
    +2\ss{\ll_3}|M(t)|\cdot|h(t)|-k_1|X(t)-Y(t)|^2\bigg\}\d t\\
 &\le
  \bigg\{k_2\int_{-\tau}^0 |X(t+\theta)-Y(t+\theta)|^2\LL(\d\theta)
 +\ff{\ll_3}\dd |h(t)|^2 +\dd |M(t)|^2 -k_1|X(t)-Y(t)|^2\bigg\}\d t\end{split}\end{equation*} for $\dd>0.$ Thus, for any $\ll\ge 0$,
\beg{equation}\label{WF0}\beg{split} &\d\{\e^{-\ll t}|M(t)|^2\}- 2\e^{-\ll t}\<M(t), \{\si(X_t)-\si(Y_t)\}\d\tt W(t)\>\\
  &\le  \e^{-\ll t}\bigg\{k_2\int_{-\tau}^0 |X(t+\theta)-Y(t+\theta)|^2\LL(\d\theta)\\
  &\qquad +\ff{\ll_3}\dd |h(t)|^2 +(\dd -\ll)|M(t)|^2 -k_1|X(t)-Y(t)|^2\bigg\}\d t,\ \ \dd>0.\end{split}\end{equation}

(a)    Let $k_1>k_2$ and $\ll=0.$ Combining (\ref{WF0}) with Lemma
\ref{L3.0} and noting that $X_0=Y_0$, we obtain
\beg{equation*}\beg{split} 0 &\le \E_\Q\int_0^T  \e^{-\lambda t}
\bigg\{ k_2\int_{-\tau}^0 |X(t+\theta)-Y(t+\theta)|^2\LL(\d\theta)\\
&\qquad\qquad
  +\ff{\ll_3 |h(t)|^2}\dd + \dd  |M(t)|^2 -k_1|X(t)-Y(t)|^2\bigg\}\d t\\
  &\le \big\{k_2-k_1 +\dd(1+k)^2\}\int_0^T \e^{-\lambda t}\E_\Q|X(t)-Y(t)|^2\d t+\ff{\ll_3}\dd\int_0^T\E_\Q|h(t)|^2\d t.\end{split}\end{equation*}
Taking
 $$\dd= \ff{k_1-k_2}{1+(1+k)^2},$$ we arrive at
 $$\int_0^T \e^{-\lambda t}\E_\Q|X(t)-Y(t)|^2\d t\le \ff{\ll_3\{1+(1+k)^2\}^2}{(k_1-k_2)^2}\int_0^T\E_\Q|h(t)|^2\d t.$$ Since by Lemma \ref{L3.0} and $X_0=Y_0$ we have
 $\int_0^T\rr_2(X_s,Y_s)^2\d s \le \int_0^T |X(t)-Y(t)|^2\d t$, this implies (\ref{WWF0}) for $\ll=0$ and the desired constant $C(0)$.

 (b) Let $k_1\le k_2$ and $\ll> \ff{k_2-k_1}{(1-k)^2}.$ Similarly to (a), by taking
 $$\dd= \ff{k_1-k_2+\ll (1-k)^2}{1+(1-k)^2} $$ in (\ref{WF0}), we obtain
$$\E_\Q\int_0^T \e^{-\ll t} \rr_2(X_t,Y_t)^2\d t\le \E_\Q\int_0^T \e^{-\ll t} |X(t)-Y(t)|^2\d t
 \le C(\ll)\int_0^T\E_\Q|h(t)|^2\d t.$$
 Therefore, (\ref{WWF0}) holds.  \end{proof}

 \beg{lem}\label{L3.3} Assume {\bf (B1)} and {\bf (B2)}. Let
 $$\tt\rr_2(\xi,\eta)^2=  |\xi(0)-\eta(0)|^2 +\rr_2(\xi,\eta)^2.$$ Then for any $\ll\in [0,\8)\cap (\ff{k_2-k_1}{(1-k)^2},\8)$,
 $$W_{2,\rr_{2,\ll}}(\Pi_\mu,\Pi_\nu)^2\le \bigg(\tau+\ff{\ll k(1-k)\tau +k_2\tau +1+k}{\ll(1-k)^2+k_1-k_2}\bigg)W_{2,\tt\rr_2}(\mu,\nu)^2,\ \ \mu,\nu\in\scr P(\C).$$
 \end{lem}

 \beg{proof} Let $\xi,\eta$ be $\C$-valued random variables with distributions $\mu$ and $\nu$ respectively, which are independent of $W([0,\8))$ such that
 \beq\label{JG}\E\tt\rr_2(\xi,\eta)^2= W_{2,\tt\rr_2}(\mu,\nu)^2.\end{equation} By {\bf (B2)} and It\^o's formula,
 \beg{equation*}\beg{split} &\d\{\e^{-\ll t} |M(t)|^2\} - 2\e^{-\ll t} \<M(t), \{\si(X_t)-\si(Y_t)\}\d  W(t)\>\\
 &\le \e^{-\ll t} \bigg\{k_2\int_{-\tau}^0 |X(t+\theta)-Y(t+\theta)|^2\LL(\d\theta) -k_1|X(t)-Y(t)|^2 -\ll|M(t)|^2\bigg\}\d t.\end{split}\end{equation*} Then, it follows from Lemma \ref{L3.0} that
\beg{equation*}\beg{split}  \E \bigg\{|M(0)|^2 &+ \int_0^T \e^{-\ll t} \{k_2-k_1-\ll(1-k)^2\}|X(t)-Y(t)|^2\d t\\
& +\{\ll k(1-k)\tau +k_2\tau\} \rr_2(\xi,\eta)^2\bigg\} \ge 0.\end{split}\end{equation*}   Since due to (\ref{WFF})
 $$|M(0)|^2\le (1+k)|\xi(0)-\eta(0)|^2 +k(1+k)\rr_2(\xi,\eta)^2\le (1+k)\tt\rr_2(\xi,\eta)^2,$$ this implies that
 $$\E\int_0^T \e^{-\ll t} |X(t)-Y(t)|^2 \d t \le \ff {\ll k(1-k)\tau+k_2\tau+1+k}{\ll(1-k)^2+k_1-k_2}\E\tt\rr_2(\xi,\eta)^2,\ \ T>0.$$
Combining this with Lemma \ref{L3.0}(1) for $\LL(\d\theta)= \ff 1 \tau \d\theta$ on $[-\tau,0]$, we conclude that
$$W_{2,\rr_{2,\ll}} (\Pi_\mu,\Pi_\nu)^2 \le \E\int_0^\8 \e^{-\ll t} \rr_2(X_t,Y_t)^2\d t  \le
 \bigg(\tau + \ff {\ll k(1-k)\tau+k_2\tau+1+k}{\ll(1-k)^2+k_1-k_2}\bigg)\E\tt\rr_2(\xi,\eta)^2.$$ Therefore, the proof is finished according to (\ref{JG}).\end{proof}

\section{An Extension of Theorem \ref{T3.1} to neutral functional SPDEs}

In this section we shall discuss the transportation cost
inequalities for the laws of segment processes of a class of neutral
functional SPDEs in infinite-dimensional setting.  Let
$(H,\<\cdot,\cdot\>,|\cdot|)$ be a real separable Hilbert
space,   let   $\C=C([-\tau, 0]; H)$ be equipped with the uniform norm $\rr(\xi,\eta):= \|\xi-\eta\|_\infty$, and let $\rr^T_\8, \rr_2$ and $\rr_{2,\ll}$ be defined by
\eqref{eq011}, \eqref{eq07} and \eqref{eq08} respectively.  Let $\L(H)$ (resp. $\L_{HS}(H)$) be the
set of all bounded (resp. Hilbert-Schmidt) operators on $H$ equipped
with the operator norm $\|\cdot\|$ (resp. Hilbert-Schmidt norm
$\|\cdot\|_{HS}$).

Let $(A,\D(A))$ be a self-adjoint operator on $H$ with spectrum
$\si(A)\subset (-\8, -\ll_0]$ for some constant $\ll_0>0$, and let
$G,b:\C\to H$  and $\si:\C\to\scr L(H)$ be Lipschitz continuous.
Consider the neutral functional SPDE
\begin{equation}\label{eq22}
\begin{cases} \d\{Z(t)-G(Z_t)\}=\{AZ(t)+b(Z_t)\}\d t+\si(Z_t)\d
W(t),\ \ \
t\in[0,T],\\
Z_0=\xi\in C,
\end{cases}
\end{equation}
where $(W(t))_{t\ge0}$  is the cylindrical Wiener process on $H$ with
respect to  a complete probability space $(\OO, \scr {F},\P)$ with
 natural filtration $\{\F_{t}\}_{t\ge0}$. Throughout the section, we assume that   equation (\ref{eq22}) has a unique mild solution, which, by definition,
 is a continuous adapted $H$-valued  process
 $\{Z(t)\}_{t\ge -\tau}$  such that $Z_0=\xi$ and
\begin{equation*}
\begin{split}
Z(t)&=\e^{tA}\{\xi(0)-G(\xi)\}+G(Z_t)+\int_0^tA\e^{(t-s)A}G(Z_s)\d s\\
&+\int_0^t\e^{(t-s)A}b(Z_s)\d s+\int_0^t\e^{(t-s)A}\si(Z_s)\d W(s),\ \ t\ge 0
\end{split}
\end{equation*} holds. For concrete conditions implying the existence and uniqueness of mild solution, we   refer  to e.g.   \cite[Theorem 3.2]{BH} and \cite[Theorem 6]{BH10}.

Let $\Pi_\mu$ be the distribution of $\{Z_t\}_{t\ge 0}$ with initial
distribution $\mu$.
 To establish the transportation cost inequality, we further need the following   conditions.

\begin{enumerate}
\item[\textmd{\bf(C1)}] There exist constants $\bar{\ll}_1\in\R$ and $ \bar{\ll}_2\ge 0$ such that
\beg{equation*}
\begin{split}
&2\big\<\xi(0)-\eta(0)+G(\eta)-G(\xi),A\xi(0)-A\eta(0)+
b(\xi)-b(\eta)\big\>\\
&\quad\ \ \ \ \ \ \ \ \ \ \ \ \ \  +\|\si(\xi)-\si(\eta)\|^2_{HS}
\le\bar{\ll}_1\|\xi-\eta\|_\8,\\
&\|\si(\xi)-\si(\eta)\|_{HS}^2\le\bar{\ll}_2\|\xi-\eta\|_\8,
\end{split}
\end{equation*}
for $\xi,\eta\in\C$ with $\xi(0),\eta(0)\in \D(A)$.

\item[\textmd{\bf(C2)}] There exist constants $\bar{\kk}_1\in\R,\bar{\kk}_2\ge0$ and a probability
measure $\bar{\LL}$ on $[-\tau,0]$ such that \beg{equation*}
\begin{split}2\big\<\xi(0)-\eta(0)&+G(\eta)-G(\xi),A\xi(0)-A\eta(0)+
b(\xi)-b(\eta)\big\>+\|\si(\xi)-\si(\eta)\|^2_{HS}\\
&\le-\bar{\kk}_1|\xi(0)-\eta(0)|^2+\bar{\kk}_2\int_{-\tau}^0|\xi(\theta)-\eta(\theta)|^2\bar{\LL}(\d
\theta)
\end{split}
\end{equation*}
for $\xi,\eta\in\C$ with $\xi(0),\eta(0)\in \D(A)$.

\end{enumerate}

Obviously, {\bf (C1)} (resp. {\bf (C2)}) holds provided $b,\si$ and $AG$ (i.e. $G$ takes vale in $\D(A)$) are Lipschitz continuous w.r.t. 
$\rr$ (resp. $\rr_2$). 

\

Let $\xi\in \C$ and $T>0$ be fixed, and as before let $\Pi_\xi^T$ denote the law of $Z_{[0,T]}:=(Z_t)_{t\in[0,T]}.$
For any   $F\ge0$ such that $\Pi^T_\xi(F)=1$, let
$\Q, m(t)$ be defined in the proof of Lemma \ref{Lemma 1} with $X_{[0,T]}$ replaced by
$Z_{[0,T]}$.
 For the $H$-valued $\F_t$-Brownian motion $\tt{W}$
defined by \eqref{eq06} and on the probability space $(\OO,\F,\Q)$,
\eqref{eq22} can be rewritten as
\begin{equation}\label{eq23}
\begin{cases}
\d\{Z(t)+G(Z_t)\}=\{AZ(t)+b(Z_t)+\si(Z_t)h(t)\}\d t+\si(Z_t)\d \tt{W}(t),\\
Z_0=\xi.
\end{cases}
\end{equation}
Consider the following equation
\begin{equation}\label{eq24}
\begin{cases}
\d\{Y(t)+G(Y_t)\}=\{AY(t)+b(Y_t)\}\d t+\si(Y_t)\d \tt{W}(t),\\
Y_0=\xi.
\end{cases}
\end{equation}
Then $\tt{M}(t):=Z(t)-Y(t)+G(Y_t)-G(Z_t)$ solves the following
equation
\begin{equation}\label{eq25}
\begin{cases}
\d\tt{M}(t)=\{A(Z(t)-Y(t))+b(Z_t)-b(Y_t)+\si(Z_t)h(t)\}\d t\\
\ \ \ \ \ \ \ \ \ \ \ \ \ \ +(\si(Z_t)-\si(Y_t))\d \tt{W}(t),\\
Y_0=Z_0.
\end{cases}
\end{equation}
Then  repeating the proofs  of Theorem \ref{Theorem 1} and Theorem
\ref{T3.1} respectively, we obtain the following results.

\beg{thm}\label{T4.2} Assume {\bf (A1)},{\bf(A3)} and {\bf (C1)}.
Let $\mu\in\scr P(\C)$ and $F$ be  non-negative measurable function
$F$ on $C([0,\8);\C)$ such that $\Pi_\mu(F)=1.$ Then
 \begin{equation*}
W_{2,\rr_\8^T}(F\Pi^T_\mu,\Pi^T_\mu)\le\ss{\bb(T)}W_{2,\rr}(\mu,\mu^T_F)+\ss{\aa(T)}\ss{\Pi_\mu^T(F\log
F)},
\end{equation*}
where $\aa(T)$ and $\bb(T)$ are defined by \eqref{eq014} and
\eqref{eq015} with $\ll_1$ and $\ll_2$ replaced by $\bar{\ll}_1$ and
$\bar{\ll}_2$ respectively.
\end{thm}

\beg{thm}\label{T4.1}  Assume {\bf (B1)}, {\bf(C2)} and {\bf (A3)}.
    Let $\mu\in\scr P(\C)$ and $F$ be  non-negative
measurable function $F$ on $C([0,\8);\C)$ such that $\Pi_\mu(F)=1.$
\beg{enumerate} \item[$(1)$]  If $\bar{\kk}_1>\bar{\kk}_2$, then
\beg{equation*}\beg{split} W_{2,\rr_{2,0}}(\Pi_\mu,F\Pi_\mu)\le &\ff{\ss{2\ll_3}\{1+(1+\kk)^2\}}{\bar{\kk}_1-\bar{\kk}_2}\ss{\Pi_\mu(F\log F)}\\
 &+\ss{\tau +\ff{\bar{\kk}_2\tau+1+\kk}{\bar{\kk}_1-\bar{\kk}_2}}\,W_{2,\rr_2}(\mu,\mu_F).\end{split}\end{equation*}
\item[$(2)$] If $\bar{\kk}_1\le \bar{\kk}_2$, then for any $\ll> \ff{\bar{\kk}_2-\bar{\kk}_1}{(1-\kk)^2},$
\beg{equation*}\beg{split}W_{2,\rr_{2,\ll}}(\Pi_\mu,F\Pi_\mu)\le &\ff{\ss{2\ll_3}\{1+(1+\kk)^2\}}{\bar{\kk}_1-\bar{\kk}_2+\ll (1-\kk)^2}\ss{\Pi_\mu(F\log F)}\\
 &+\ss{\tau +\ff{\ll \kk(1-\kk)\tau +\bar{\kk}_2\tau+1+\kk}{\ll(1-\kk)^2+\bar{\kk}_1-\bar{\kk}_2}}\,W_{2,\rr_2}(\mu,\mu_F).\end{split}\end{equation*}\end{enumerate} \end{thm}


\begin{thebibliography}{17}

\bibitem{BGL} S. Bobkov, I. Gentil, M. Ledoux, Hypercontractivity of Hamilton-Jacobi
equations, {\it J. Math. Pure Appl.}, {\bf 80} (2001), 669--696.

\bibitem{BG99} S. Bobkov, F. G\"otze, Exponential integrability and
transportation cost related to logarithmic Sobolev inequalities,
{\it J. Funct. Anal.}, {\bf 163} (1999), 1--28.


\bibitem{BH10} B. Boufoussi,  S. Hajji,  Successive approximation of neutral
functional stochastic differential equations with jumps, {\it
Statist. Probab. Lett.}, {\bf80} (2010), 324--332.

\bibitem{BH} B. Boufoussi,  S. Hajji,  Successive approximation of neutral
functional stochastic differential equations in Hilbert spaces, {\it
Ann. Math. Blaise Pascal}, {\bf17} (2010), 183--197.


\bibitem{DGW04} H. Djellout, A. Guilin, L. Wu,  Transportation cost-information
inequalities for random dynamical systems and diffsions,  {\it Ann.
Probab.}, {\bf32} (2004), 2702--2732.

\bibitem{FS07} S. Fang, J. Shao, Optimal transport maps for Monge-Kantorovich
problem on loop groups,  {\it J. Funct. Anal.}, {\bf248} (2007),
225--257.


\bibitem{GL07} N. Gozlan, C. L\'{e}onard, A large deviation approach to some
transportation cost inequalities,  {\it Probab. Theory Related
Fields}, {\bf 139} (2007), 235--283.



\bibitem{GRS11} N. Gozlan, C. Roberto, P.-M. Samson,  A new characterization
of Talagrand's transport-entropy inequalities and applications, {\it
Ann. Probab.}, {\bf 39} (2011), 857--880.



\bibitem{IW89} N. Ikeda, S. Watanabe, \emph{Stochastic Differential Equations and Diffusion Processes (Second Edition),} North-Holland, New York, 1989.



\bibitem{L} M. Ledoux,  The Concentration of Measure Phenomenon, {\it Mathematical
Surveys and Monographs,} American Mathematical Society, Providence,
2001.


\bibitem{M10} Y. Ma, Transportation inequalities for stochastic
differential equations with jumps,  {\it Stochastic Process. Appl.},
{\bf120} (2010), 2--21.

\bibitem{M97}X. Mao, \emph{ Stochastic Differential Equations and Applications}, Horwood, England, 1997.





\bibitem{OV} F. Otto, C. Villani,  Generalization of an inequality by Talagrand
and links with the logarithmic Sobolev inequality,  {\it J. Funct.
Anal.}, {\bf 173} (2000), 361--400.



\bibitem{P11} S. Pal, Concentration for multidimensional diffusions and their
boundary local times, {\it Probab. Theory Relat. Fields}, in press.

\bibitem{S11} B. Saussereau, Transportation inequalities for stochastic differential equations driven by a fractional Brownian
motion,  {\it Bernoulli}, in press.



\bibitem{T96} M. Talagrand, Transportation cost for Gaussian and other product measures, {\it Geom. Funct.
Anal.}, {\bf 6} (1996), 587--600.


\bibitem{U} A.S. \"Ust\"unel, \emph{Introduction to Analysis on Wiener
space}, Lecture Notes in Math., Springer, 1995.


\bibitem{U10} A.S. \"Ust\"unel, Transport cost inequalities for
diffusions under uniform distance, arXiv:1009.5251 v3.


\bibitem{V03} C. Villani, \emph{ Topics in Optimal Transportation},  Graduates Studies in
Mathematics {\bf 58},  Providence RI: Amer. Math. Soc., 2003.

\bibitem{Vs10} Max-K. von Renesse, M. Scheutzow,   Existence and uniqueness of solutions of stochastic functional differential
equations,
 {\it Random Oper. Stoch. Equ.}, {\bf 18} (2010),  267--284.


\bibitem{W04} F.-Y. Wang, Transportation cost inequalities on path
spaces over Riemannian manifolds, {\it Illinois J. Math}, {\bf46}
(2002), 167--1206.

\bibitem{W04b} F.-Y. Wang, Probability distance inequalities on Riemannian manifolds and path
spaces, {\it J. Funct. Anal.}, {\bf 206} (2004), 167--190.

\bibitem{W11} F.-Y. Wang, Transportation-cost inequalities on path
spaces over manifolds with boundary, to appear in {\it Docum. Math.}

\bibitem{W10} L. Wu, Transportation inequalities for stochastic differential
equations of pure jumps, {\it Ann. Inst. Henri Poincar\'{e} Probab.
Stat.}, {\bf 46} (2010), 465--479.


\bibitem{WZ04} L. Wu,  Z. Zhang,  Talagrand's $T_2$-transportation inequality w.r.t. a
uniform metric for diffusions, {\it Acta Math. Appl. Sin. Engl.
Ser.}, {\bf20} (2004), 357--364.


 \bibitem{WZ06} L. Wu,  Z. Zhang, Talagrand's $T_2$-transportation inequality and log-Sobolev inequality for dissipative SPDEs and
 applications to reaction-diffusion equations, {\it Chinese Ann. Math. Ser. B}, {\bf 27} (2006), 243--262.











\end{thebibliography}
\end{document}